\documentclass[12pt]{article}
\usepackage{e-jcCHANGE}

\usepackage{amsthm,amsmath,amssymb}

\usepackage{graphicx}

\numberwithin{equation}{section}
\theoremstyle{plain}
\newtheorem{theorem}{Theorem}[section]
\newtheorem{lemma}[theorem]{Lemma}
\newtheorem{corollary}[theorem]{Corollary}

\theoremstyle{definition}

\theoremstyle{remark}

\title{\bf On the strength of connectedness of a random hypergraph}

\author{Daniel Poole\thanks{The author gratefully acknowledges support from NSF grant \# DMS-1101237.}\\
\small Department of Mathematics\\[-0.8ex]
\small The Ohio State University\\[-0.8ex] 
\small Columbus, Ohio, U.S.A.\\
\small\tt poole@math.osu.edu
}

\begin{document}

\maketitle


\begin{abstract}
Bollob\'{a}s and Thomason (1985) proved that for each $k=k(n) \in [1, n-1]$, with high probability, the random graph process, where edges are added to vertex set $V=[n]$ uniformly at random one after another, is such that the stopping time of having minimal degree $k$ is equal to the stopping time of becoming $k$-(vertex-)connected. We extend this result to the $d$-uniform random hypergraph process, where $k$ and $d$ are fixed. Consequently, for $m=\frac{n}{d}(\ln n +(k-1)\ln \ln n +c)$ and $p=(d-1)! \frac{\ln n + (k-1) \ln \ln n +c}{n^{d-1}}$, the probability that the random hypergraph models $H_d(n, m)$ and $H_d(n, p)$ are $k$-connected tends to $e^{-e^{-c}/(k-1)!}.$

\end{abstract}

\section{Introduction}

Let $H_d(n, p)$ denote the random $d$-uniform hypergraph with vertex set $[n]:=\{1,2,...,n\}$, where each of the ${n \choose d}$ potential (hyper)edges of cardinality $d$ is present with probability $p$, independently of all other potential edges. Likewise, let $H_d(n, m)$ be the random $d$-uniform hypergraph on $[n]$, where $m$ edges are chosen uniformly at random among all sets of $m$ potential edges. The model $H_d(n, m)$ can be gainfully viewed as a snapshot of the random hypergraph process $\{ H_d(n, \mu) \}_{\mu=0}^{n \choose d},$ where $H_d(n, \mu+1)$ is obtained from $H_d(n, \mu)$ by inserting an extra edge chosen uniformly at random among all ${n \choose d}-\mu$ remaining potential edges. For $d=2$, these models are the typical random graph models, $G(n, p)$, $G(n, m)$ and $\{G(n, \mu)\}_{\mu=0}^{{n \choose 2}}.$ 

As customary, we say that for a given $m=m(n)$ ($p$ resp.) some graph property $\mathcal{Q}$ holds {\em with high probability}, denoted {\em w.h.p.}, if the probability that $H_d(n,m)$ ($(H_d(n,p)$ resp.) has property $\mathcal{Q}$ tends to 1 as $n \to \infty$. Further, $m(n)$ is the {\em sharp threshold} for $\mathcal{Q}$ if for each $\epsilon>0$ (fixed), w.h.p. $H_d(n, (1-\epsilon)m)$ does not have $\mathcal{Q}$ and w.h.p $H_d(n, (1+\epsilon)m)$ does have $\mathcal{Q}$. For the random hypergraph process, the {\em stopping time} of $\mathcal{Q}$, denoted $\tau(\mathcal{Q})$ is the first moment that the process has this property $\mathcal{Q}$; we denote the hypergraph process stopped at this time by $H_d(n, \tau(\mathcal{Q}))$.

In one of the first papers on random graphs, Erd\H{o}s and R\'{e}nyi~\cite{ER59} showed that $m=\tfrac{1}{2} n \ln n$ is a sharp threshold for connectivity in $G(n,m)$. Later, Stepanov~\cite{stepanov} established the sharp threshold of connectivity for $G(n,p)$ among other results. More recently, Bollob\'{a}s and Thomason~\cite{bollobas thomason} proved the stronger result for the random graph process that w.h.p. the moment the graph process loses its last isolated vertex is also the moment that the process becomes connected; in other words, w.h.p. $\tau(\text{no isolated vertices}) = \tau(\text{connected})$. We prove the analogous result for the the random $d$-uniform hypergraph process; a consequence of this result is that $m=\tfrac{n}{d} \ln n$ is a sharp threshold of connectivity for $H_d(n,m)$. 

There are various measures for the {\em strength} of connectedness of a connected graph, but here we will focus on $k$-(vertex-)connectivity.  For $k\in \mathbb{N}$, a hypergraph with more than $k$ vertices is $k$-{\em connected} if whenever $k-1$ vertices are deleted, along with their incident edges, the remaining hypergraph is connected. Note that the definition of 1-connectedness coincides with connectedness. Necessarily, for a hypergraph to be $k$-connected, each vertex must have degree at least $k$, because if a vertex $v$ has degree less than $k$, then we can delete a neighbor from each incident edge to isolate $v$. However, as commonly seen in these types of results, the main barrier to $k$-connectivity in these random graph models arises from such vertices that can be separated from the rest of the graph by the deletion of their neighbors (see for instance Erd\H{o}s-R\'{e}nyi~\cite{erdos renyi}, Ivchenko~\cite{ivchenko}, Bollob\'{a}s~\cite{bollobas4},\cite{bollobas book}). Here, we extend this idea to random $d$-uniform hypergraphs; in particular, we find that  if $m_0=\tfrac{n}{d} (\ln n +(k-1)\ln \ln n -\omega)$ and $m_1=\tfrac{n}{d}(\ln n +(k-1)\ln \ln n +\omega)$, $\omega\to\infty$, then w.h.p. $H_d(n,m_0)$ is not $k$-connected and w.h.p. $H_d(n, m_1)$ is $k$-connected; also we find an analogous threshold value for $H_d(n,p)$.

A stronger result concerns the random graph process where edges are added one after another. Let $\tau_k := \tau(\text{min degree at least }k)$ and $T_k:=\tau(k\text{-connected})$; note that $\tau_k \leq T_k$. In \cite{bollobas thomason}, Bollob\'{a}s and Thomason showed that for $d=2$ (the graph case) and any $k=k(n) \in [1, n-1]$, $P(\tau_k = T_k) \to 1$. We extend this result for $d$-uniform random hypergraphs albeit for {\em fixed} $d$ and $k$.

\begin{theorem}\label{theorem: 1.25}
W.h.p, at the moment the $d$-uniform hypergraph process loses its last vertex with degree less than $k$, this process becomes $k$-connected. Formally, for $d \geq 3$ and  $k \geq 1$ (both fixed), $P(\tau_k=T_k) \to 1$. 
\end{theorem}

To prove this result, we begin by determining the likely range of $\tau_k$, and further that just prior to this window, at some $m_0$ edges, w.h.p. there are not many vertices of degree less than $k$. Then, we prove that w.h.p. $H_d(n, m_0)$ is almost $k$-connected in the sense that whenever $k-1$ vertices are deleted, there is a massive component using almost all leftover vertices. Third, we show that w.h.p. to isolate a vertex of $H_d(n, \tau_k),$ you would have to delete at least $k$ of its neighbors (this is trivially true for graphs, but not so for $d \geq 3$). In particular, we show that w.h.p. edges incident to degree $k-1$ vertices have trivial intersection (just the vertex itself). Finally, we show that the probability that $\tau_k < T_k$, but these three previous likely events also hold tends to zero, which completes the proof of the theorem. The following corollary is nearly immediate in light of the theorem.

\begin{corollary}\label{cor: 1.25}
\textbf{(i)} Let $m=\frac{n}{d}\left(\ln n + (k-1)\ln \ln + c_n\right)$, where $c_n \to c \in \mathbb{R}$. W.h.p. $H_d(n, m)$ is $(k-1)$-connected, but not $(k+1)$-connected. Further, the probability that $H_d(n, m)$ is $k$-connected tends to $e^{-e^{-c}/(k-1)!}$. \\
\textbf{(ii)} Let $p=(d-1)! \frac{\ln n + (k-1)\ln \ln + c_n}{n^{d-1}}$, where $c_n \to c \in \mathbb{R}$. W.h.p. $H_d(n, p)$ is $(k-1)$-connected, but not $(k+1)$-connected. Further, the probability that $H_d(n, p)$ is $k$-connected tends to $e^{-e^{-c}/(k-1)!}$. 
\end{corollary}

For the remainder of this paper, let $d \geq 3$ and $k \geq 1$ be fixed numbers.

\section{Likely range of $\tau_k$}

\begin{lemma}\label{lemma: 6.5.1}
Let $\omega=\omega(n) \to \infty$, but $\omega = o(\ln \ln n)$,  $m_0= \frac{n}{d} \left( \ln n + (k-1) \ln \ln n - \omega\right)$ and $m_1 = \frac{n}{d} \left( \ln n + (k-1) \ln \ln n + \omega \right)$. Then w.h.p.,

\textbf{(i)} the minimum degree of $H_d(n, m_0)$ is $k-1$ and the number of vertices with degree $k-1$ is in the interval
\begin{equation}\label{eqn: 6.4.3}
\left[ \frac{1}{2} \frac{e^\omega}{(k-1)!}, \, \frac{3}{2} \frac{e^\omega}{(k-1)!} \right].
\end{equation}

\textbf{(ii)} there are no vertices of degree $k-1$ in $H_d(n, m_1)$. 

Consequently, w.h.p. $\tau_k \in [m_0, m_1]$. 
\end{lemma}

\begin{proof}
We prove that the number of vertices with degree $k-1$, denoted by $X$, is in the interval \eqref{eqn: 6.4.3} by Chebyshev's Inequality. Note that a given vertex can be in ${n-1 \choose d-1}$ possible edges, so
\begin{equation*}
E[X] = n P[\text{deg}(1) = k-1] = n \frac{ {{n-1 \choose d-1} \choose k-1} { {n \choose d}- {n-1 \choose d-1} \choose m_0 - k +1 } }{{{n \choose d} \choose m_0}}.
\end{equation*}
Here and elsewhere in this paper, we use the identity ${N \choose m-\ell} = {N \choose m} \tfrac{(m)_\ell}{(N-m+\ell)_\ell}$, where $(j)_\ell=j(j-1)\cdots (j-\ell+1),$ and later, we use the inequality ${N \choose m-\ell} \leq {N \choose m}\left( \tfrac{m}{N-m}\right)^\ell$. Now
\begin{equation*}
E[X] = (1+O(1/n))  \frac{n\cdot  n^{(d-1)(k-1)}}{(k-1)! ( (d-1)!)^{k-1}} \left( \frac{d! \, m_0}{n^d} \right)^{k-1} \frac{ {{n \choose d} - {n-1 \choose d-1} \choose m_0}}{{{n \choose d} \choose m_0}}.
\end{equation*}
This latter fraction can be sharply approximated. 
\begin{align}\label{eqn: 7.28.2}
\frac{ {{n \choose d} - {n-1 \choose d-1} \choose m_0}}{{{n \choose d} \choose m_0}} &= \prod_{i=0}^{m_0-1} \left( 1 - \frac{{n-1 \choose d-1}}{{n \choose d}-i} \right)= \prod_{i=0}^{m_0-1} \left( 1 - \frac{d}{n} +O\left( \frac{i}{n^{d+1}} \right)\right)\nonumber \\ & = \exp \left( \sum_{i=0}^{m_0-1} \left[ \frac{-d}{n}+O\left(\frac{1}{n^2}\right)+O\left(\frac{i}{n^{d+1}}\right) \right] \right)   \nonumber \\ &=\left( 1 +O\left(\frac{\ln n}{n} \right) \right) \exp \left( - \frac{d m_0}{n} \right) = \left( 1 + O\left(\frac{\ln n}{n} \right)\right) \frac{e^\omega}{n (\ln n)^{k-1}}.
\end{align}
Hence
\begin{align*}
E[X] &= \left( 1 + O\left( \frac{\ln n}{n} \right)\right) \frac{e^\omega}{(k-1)!} \left( \frac{ d m_0/n }{\ln n} \right)^{k-1} \\ &= \left( 1 + O\left( \frac{\ln \ln n}{\ln n} \right)\right) \frac{e^\omega}{(k-1)!}.
\end{align*}

For the second factorial moment, we have that
\begin{equation*}
E[X(X-1)] = n (n-1) P(\text{deg}(1)=\text{deg}(2) = k-1).
\end{equation*}
We break this latter probability over $i$, the number of edges that include both vertices $1$ and $2$. In particular, vertex 1 is in $k-1-i$ edges that do not contain vertex $2$ and vice versa; further there are $m_0 - 2(k-1)+i$ edges that include neither vertex 1 or 2. Since there are ${n -2 \choose d-2}$ potential hyperedges containing both vertices and ${n-1\choose d-1}-{n-2 \choose d-2}$ potential hyperedges containing one vertex but not the other, we have that 
\begin{equation}\label{eqn: 7.28.3}
P(\text{deg}(1)=\text{deg}(2)=k-1) = \sum_{i=0}^{k-1}{{n-2 \choose d-2} \choose i} {{n-1 \choose d-1} - {n-2 \choose d-2} \choose k-1-i}^2 \frac{ {{n \choose d}-2{n-1 \choose d-1}+{n-2 \choose d-2} \choose m_0 - 2(k-1)+i }}{{{n \choose d} \choose m_0}}.
\end{equation}
Just as in \eqref{eqn: 7.28.2}, we can estimate this latter fraction
\begin{align*}
\frac{ {{n \choose d}-2{n-1 \choose d-1}+{n-2 \choose d-2} \choose m_0 - 2(k-1)+i }}{{{n \choose d} \choose m_0}} &= \frac{(m_0)_{2(k-1)-i}}{ {\scriptstyle \left( {n \choose d}-2{n-1 \choose d-1}+{n-2 \choose d-2} - m_0 + 2(k-1)-i\right)_{2(k-1)-i} } } \times \frac{{{n \choose d}-2{n-1 \choose d-1}+{n-2 \choose d-2} \choose m_0}}{{{n \choose d} \choose m_0}} \\ &= \left( 1 + O \left( \frac{\ln \ln n}{\ln n } \right) \right) \left( \frac{m_0}{{n \choose d}} \right)^{2(k-1)-i} \frac{e^{2 \omega}}{n^2 \, (\ln n)^{2(k-1)}}.
\end{align*}
Using these asymptotics, one can show that the $i$'th term in \eqref{eqn: 7.28.3} is on the order of $\frac{e^{2\omega}}{n^{2+i} (\ln n)^i}$. In particular, the sum of the terms over $i \in [1, k-1]$, is $O(n^{-3})$.  Therefore
\begin{align*}
P(\text{deg(1)= deg(2)= }k-1) &= {{n-1 \choose d-1} - {n-2 \choose d-2} \choose k-1}^2 \frac{ {{n \choose d}-2{n-1 \choose d-1}+{n-2 \choose d-2} \choose m_0 - 2(k-1) }}{{{n \choose d} \choose m_0}} + O(n^{-3}) \\ &= \left( 1 + O\left( \frac{\ln \ln n}{\ln n} \right)\right) \frac{e^{2 \omega}}{((k-1)!)^2 n^2};  
\end{align*}
whence
\begin{equation*}
E[X(X-1)] = \left( 1 + O\left( \frac{\ln \ln n}{\ln n}\right)\right) E[X]^2.
\end{equation*}
Consequently,
\begin{equation*}
\text{var}[X] = E[X] + O\left( (E[X])^2 \frac{\ln \ln n}{\ln n} \right).
\end{equation*}
By Chebyshev's Inequality, $X$ is concentrated around its mean and in particular, w.h.p. $X$ is in the interval \eqref{eqn: 6.4.3}. To finish the proof of part \textbf{(i)}, it remains to show that w.h.p. there are no vertices of degree less than $k-1$, which can be done by a first moment argument using similar techniques to the asymptotics of $E[X]$. Similarly, for part \textbf{(ii)}, one can easily show that the expected number of vertices of degree $k-1$ in $H_d(n, m_1)$ tends to zero as well. 
\end{proof}

\section{$H_d(n, m_0)$ is almost $k$-connected}

Now we will establish that w.h.p. $H_d(n, m_0)$ is almost $k$-connected in the sense that if $k-1$ vertices are deleted, then there remains a massive component containing almost all left-over vertices. To this end, we prove an analogous statement for the random Bernoulli hypergraph $H_d(n, p)$ and use a standard conversion lemma to obtain the desired result for $H_d(n, m_0)$. In this next lemma, we pick a specific version of $m_0$, one where $\omega=\ln \ln \ln n$. 

\begin{lemma}\label{lemma: 7.28.1}
Let $m'_0 = \frac{n}{d} \left( \ln n  + (k-1)\ln \ln n- \ln \ln \ln n \right)$ and $p=m'_0/{n \choose d}$. With high probability, 

\textbf{(i)} $H_d(n, p)$ has the property ``whichever $k-1$ vertices are deleted, there remains a giant component which includes all but up to $\ln n$ leftover vertices."

\textbf{(ii)} $H_d(n, m'_0)$ has the property ``whichever $k-1$ vertices are deleted, there remains a giant component which includes all but up to $\ln n$ leftover vertices."
\end{lemma}

\begin{proof}
\textbf{(i)} Given a set of $k-1$ vertices, $\mathbf{v} = \{v_1, \ldots, v_{k-1} \}$, let $\mathcal{F}(\mathbf{v})$ be the event that if the vertices $\mathbf{v}$ are deleted from $H_d(n, p)$ along with their incident edges, then there remains no components of size at least $n-(k-1)-\ln n$. In particular, we wish to show that w.h.p. $H_d(n, p)$ is not in $\mathcal{F}(\mathbf{v})$ for any $\mathbf{v}$. Using the union bound over all $k-1$ element sets of $[n]$ as well as symmetry, we find that
\begin{equation}\label{eqn: 6.3.5}
P\left( \mathop{\cup}_{\mathbf{v}} \mathcal{F}(\mathbf{v}) \right) \leq {n \choose k-1} P(\mathcal{F}(\mathbf{v}^*)),
\end{equation}
where $\mathbf{v}^*=\{n-(k-1)+1, \ldots, n-1, n\}$. Note that the remaining hypergraph left after deleting $\mathbf{v}^*$ from $H_d(n, p)$ is distributed as $H_d(n', p), \, n':=n-(k-1)$ (this is the primary reason that we consider the Bernoulli hypergraph $H_d(n, p)$ rather than $H_d(n, m)$). Therefore $P(\mathcal{F}(\mathbf{v}^*))$ is precisely the probability that $H_d(n', p)$ does not have a component of size at least $n'- \ln n$. 

To bound $P(\mathcal{F}(\mathbf{v}^*))$, we note that any hypergraph on $n'$ vertices without a component of size at least $n'-\ln n$ has a set of vertices $S$ such that there are no edges between $S$ and $[n'] \setminus S$ where $|S| \in [\ln n, n'-\ln n].$ To see this fact, consider a hypergraph $H$ on $n'$ vertices without such a large component and let $L_1, \ldots, L_{\ell}$ be the vertex sets of the components of $H$ in increasing order by their cardinalities. Then there is some minimal $j$ so that 
\begin{equation*}
\ln n \leq \big | \cup_{i=1}^j L_i \big | < \ln n + n'- \ln n = n'.
\end{equation*}
Further, $L_{j+1}$ is not empty and since $|L_j| \leq |L_{j+1}|$, we have that $|L_j| \leq n'/2$ and 
\begin{equation*}
\ln n \leq \big | \cup_{i=1}^j L_i \big | < \ln n + n'/2 < n' - \ln n.
\end{equation*}
Clearly, there are no edges including a vertex of $S:=\cup_{i=1}^j  L_i$ and $[n'] \setminus S$. 

Therefore, 
\begin{equation*}
P(\mathcal{F}(\mathbf{v}^*)) \leq \sum_{s=\ln n}^{n'-\ln n} P(\exists S \subset [n'], |S|=s, \text{ no edge between }S\text{ and }[n']\setminus S),
\end{equation*}
and by symmetry over all such vertex sets $S$,
\begin{equation*}
P(\mathcal{F}(\mathbf{v}^*))  \leq \sum_{s=\ln n}^{n'-\ln n} {n' \choose s} P(\text{no edge between }[s]\text{ and }[n'] \setminus [s]).
\end{equation*}
Further, this latter probability is symmetric about $s=n'/2$ (i.e. the probabilities corresponding to $s$ and $n'-s$ are equal). Hence
\begin{equation*}
P(\mathcal{F}(\mathbf{v}^*)) \leq 2 \sum_{s=\ln n}^{\lfloor n'/2 \rfloor} {n' \choose s}P(\text{no edge between }[s]\text{ and }[n'] \setminus [s]).
\end{equation*}

The number of potential edges that contain at least one vertex from $[s]$ and at least one vertex from $[n']\setminus[s]$ is ${n' \choose d}-{s \choose d}-{n'-s \choose d}$. Hence
\begin{align*}
P(\mathcal{F}(\mathbf{v}^*)) \leq 2\sum_{s=\ln n}^{\lfloor n'/2 \rfloor} {n' \choose s} (1-p)^{{n' \choose d}-{s \choose d}-{n'-s \choose d}} &\leq 2\sum_{s=\ln n}^{\lfloor n'/2 \rfloor} {n' \choose s} e^{-p\left({n' \choose d}-{s \choose d}-{n'-s \choose d}\right)} \\& =: 2 E_1 + 2 E_2,
\end{align*}
where $E_1$ and $E_2$ are the sums over $S_1 := [\ln n, n/(\ln n)]$ and $S_2 := (n/(\ln n), \lfloor n'/2 \rfloor]$ respectively. We begin with analyzing $E_2$ since these bounds will be cruder and simpler.  

Trivially, 
\begin{equation}\label{eqn: 6.4.1}
E_2  \leq  \sum_{s \in S_2} {n' \choose s} e^{ - p {n' \choose d} + p \max\limits_{t \in S_2} \left( {t \choose d}+{n'-t \choose d} \right)} \leq 2^{n'} e^{ -p{n' \choose d} + p \max \limits_{t \in S_2}  \left( {t \choose d}+{n'-t \choose d} \right)}.
\end{equation}
Now let's take on these binomial coefficient terms. Trivially ${\nu \choose d}\leq \frac{\nu^d}{d!},$ and so
\begin{equation*}
{t \choose d}+{n'-t \choose d} \leq \frac{t^d+(n'-t)^d}{d!}.
\end{equation*}
Further, the function $f(t)=t^d+(n'-t)^d$ is decreasing for $t \in S_2$, so we have that
\begin{equation*}
 \max_{t \in S_2} \left( {t \choose d}+{n'-t \choose d} \right) \leq \frac{1}{d!} \left( \left(\frac{n}{\ln n}\right)^d + \left(n'-\frac{n}{\ln n}\right)^d \right).
\end{equation*}
Now our bound in \eqref{eqn: 6.4.1} becomes
\begin{equation*}
E_2 \leq 2^{n'} \exp \left( -p {n' \choose d} + \frac{p}{d!} \left(\frac{n}{\ln n}\right)^d + \frac{p}{d!}\left(n'-\frac{n}{\ln n}\right)^d  \right).
\end{equation*}
In the previous expression, the leading order terms in the first and third terms will cancel and the middle term is absorbed in the error. Namely, we have that
\begin{equation*}
E_2 \leq 2^{n'} \exp \left( -\frac{p(n')^d}{d!} + O(\ln n)   + O\left( \frac{n \, \ln n}{(\ln n)^d} \right) +  \frac{p(n')^d}{d!}  - \frac{p \, d}{d!}  \frac{n^d}{\ln n} + O\left(\frac{n}{\ln n}\right)      \right),
\end{equation*}
or equivalently
\begin{align}\label{eqn: 6.3.4}
E_2 &\leq 2^{n'} \exp \left( - \frac{p \, n^d}{(d-1)! \, \ln n} + O\left(\frac{n}{\ln n}\right)\right) \nonumber  \\ & = \exp \left( (\ln 2 - 1)n + o(n) \right) \leq e^{-n/4}.
\end{align}

Now let's take on the sum $E_1$. We begin with taking the leading order terms in the exponent. 
\begin{equation}\label{eqn: 6.3.1}
E_1 \leq \sum_{s \in S_1} {n' \choose s} \exp \left( -p {n' \choose d} \left( 1 - \frac{{n' -s \choose d}}{{n' \choose d}} - O\left(\frac{s^d}{n^d}\right) \right)\right).
\end{equation}
Uniformly over $s \in S_1$, we have that 
\begin{equation*}
\frac{{n' - s \choose d}}{{n' \choose d}} = \left( 1 - \frac{s}{n'} +O\left(\frac{s}{n^2}\right) \right)^d = 1 - \frac{d s}{n'} + O\left( \frac{s^2}{n^2}\right).
\end{equation*}
Consequently, the exponent in \eqref{eqn: 6.3.1} is
\begin{equation*}
p{n' \choose d} \left( \frac{ds}{n}+O\left(\frac{s^2}{n^2}\right)\right) = \left( \ln n + (k-1)\ln\ln n - \ln \ln \ln n\right)\left( s+O(s^2/n)\right)+O(1);
\end{equation*}
whence there is some (fixed) $\gamma>0$ such that for all $s \in S_1$,
\begin{equation}\label{eqn: 6.3.2}
p{n' \choose d}\left(\frac{ds}{n} + O\left( \frac{s^2}{n^2}\right)\right) \geq (\ln n - \ln \ln \ln n)(s-\gamma s^2/n) - \gamma.
\end{equation}
Using the bound ${n' \choose s} \leq (e n/s)^s$ as well as \eqref{eqn: 6.3.2}, \eqref{eqn: 6.3.1} becomes
\begin{align*}
E_1 &\leq \sum_{s \in S_1} \left( \frac{en}{s} \right)^s \exp \left( -(\ln n - \ln \ln \ln n)(s-\gamma s^2/n)+\gamma \right) \\ & = e^{\gamma} \sum_{s \in S_1} \exp \left( s \left( 1 - \ln s + \ln \ln \ln n + \gamma s (\ln n - \ln \ln \ln n)/n \right) \right).
\end{align*}
However, for $s \in S_1$, we have that 
\begin{equation*}
s \leq \frac{n}{\ln n} \leq \frac{2 n}{\ln n - \ln \ln \ln n} \implies s (\ln n - \ln \ln \ln n)/n \leq 2.
\end{equation*}
Therefore
\begin{equation*}
E_1 \leq e^{\gamma} \sum_{s \in S_1} \exp \left( s \left( 1 - \ln s + \ln \ln \ln n + 2\gamma\right)\right).
\end{equation*}
This sum is dominated by the first term ($s=\ln n$) because ratios of consecutive terms uniformly tend to zero. Consequently, we have that 
\begin{equation}\label{eqn: 6.3.3}
E_1 = O\left[\exp \left( - \ln n \ln \ln n + o(\ln n \ln \ln n) \right)\right].
\end{equation}

Summing our bounds for $E_1$ and $E_2$ (\eqref{eqn: 6.3.3} and \eqref{eqn: 6.3.4}, respectively), we have that
\begin{equation}\label{eqn: 6.4.2}
P(\mathcal{F}(\mathbf{v}^*)) \leq \exp \left( -\ln n \ln \ln n + o(\ln n \ln \ln n) \right),
\end{equation}
which most definitely is $o(n^{-(k-1)})$ and so by the bound \eqref{eqn: 6.3.5}, part \textbf{(i)} of the lemma is proved. 

Part \textbf{(ii)} is established by using a standard conversion technique between $H_d(n, m)$ and $H_d(n, p)$. For any hypergraph property $\mathcal{A}$, we have that
\begin{equation}\label{eqn: 3.4.1}
P(H_d(n, p) \in \mathcal{A}) = \sum_{m=0}^{{n \choose d}} P(H_d(n, m) \in \mathcal{A}) P(e(H_d(n, p))=m),
\end{equation}
where $e(H)$ is the number of edges of $H$. Therefore, for any (possible) $m$,
\begin{equation*}
P(H_d(n, p) \in \mathcal{A}) \geq P(H_d(n, m) \in \mathcal{A}) P(e(H_d(n, p))=m), 
\end{equation*}
whence
\begin{equation*}
P(H_d(n, m) \in \mathcal{A}) \leq \frac{P(H_d(n, p) \in \mathcal{A})}{P(e(H_d(n, p)=m))}.
\end{equation*}
For $m= \Theta(n \ln n)$ and $p=m/{n \choose d}$, one can show that
\begin{equation*}
P(e(H_d(n, p))=m) = {{n \choose d}\choose m} p^m (1-p)^{{n \choose d}-m}= \Theta(m^{-1/2}).
\end{equation*}
Hence in our case,
\begin{equation*}
P(H_d(n, m'_0) \in \mathop{\cup}_{\mathbf{v}} \mathcal{F}(\mathbf{v}) ) = O\left( \sqrt{n \ln n} \, P(H_d(n, p) \in \mathop{\cup}_{\mathbf{v}} \mathcal{F}(\mathbf{v})) \right).
\end{equation*}
In the proof of part \textbf{(i)}, we found that this latter probability tends to zero superpolynomially fast. 
\end{proof}

\section{Quasi-disjoint Edges}

For the random graph process ($d=2$), it was found that the main barrier to $k$-connectivity is the presence of vertices of degree less than $k$, which could be isolated with the deletion of their neighbors (see Erd\H{o}s-R\'{e}nyi~\cite{erdos renyi}, Ivchenko~\cite{ivchenko}, Bollob\'{a}s~\cite{bollobas4},\cite{bollobas book}). We will find a similar situation for the random hypergraph process. 

However, we run into an additional issue here for hypergraphs. Even if the degree of a vertex $v$ is $k$, we could isolate $v$ with the deletion of less than $k$ vertices. For instance, if all of $v$'s edges also include vertex $w$, then the deletion of just $w$ from the hypergraph (along with its incident edges) will isolate $v$ from the rest of the hypergraph. Our ultimate goal in this section is to show that w.h.p. each vertex of $H_d(n, \tau_k)$ has at least $k$ edges whose pairwise intersections are precisely $\{v\}$; in this case, for any vertex, you would need to delete at least $k$ of its neighbors to isolate it. To this end, we first prove that w.h.p. $H_d(n, m_0)$ has this property for vertices with degree at least $k$ and as nearly as could be expected for vertices with degree $k-1$. 

A set of edges $E$ incident to vertex $v$ is {\it quasi-disjoint} if all pairwise intersections of these edges are $\{v\};$ formally, if $e, f \in E, \, e \neq f$, then $e \cap f = \{v\}.$

\begin{lemma}\label{lemma: 6.5.2}
Let $m_0 = \tfrac{n}{d} ( \ln n + (k-1) \ln \ln n - \omega)$, where $\omega \to \infty,$ but $\omega=o(\ln \ln n)$. W.h.p., $H_d(n, m_0)$ is such that 

\textbf{(i)} the incident edges of a degree $k-1$ vertex form a quasi-disjoint set, 

\textbf{(ii)} vertices with degree at least $k$ have a quasi-disjoint set of incident edges with size at least $k$.
\end{lemma}

\begin{proof}
Note that both parts of this lemma are trivially true for $k=1$ and part \textbf{(i)} is also trivially true for $k=2$. Let $X(j, \ell)$ be the number of vertices whose maximum quasi-disjoint set has size $j$ and whose degree is $j+\ell$. To prove this lemma, it suffices to show that w.h.p. for $j \leq k-1$ and $\ell \geq 1$, we have that $X(j, \ell)=0$, which is shown by a first moment argument. Now
\begin{equation*}
E[X(j, \ell)] = n P(j, \ell),
\end{equation*}
where $P(j, \ell)$ is the probability that a generic vertex $v$ has a maximum quasi-disjoint set of size $j$ and whose degree is $j + \ell$. To bound this probability, note that $v$ has a set of $j$ quasi-disjoint edges and each of the remaining $\ell$ edges must have at least one vertex from the $j(d-1)$ neighbors from the quasi-disjoint edges; further the remaining $m_0-j -\ell$ edges do not include $v$. Hence
\begin{align*}
P(j, \ell) &\leq {{{n-1 \choose d-1} \choose j}} {{{j (d-1)\choose 1} {n -2 \choose d-2} \choose \ell}} \frac{{{n \choose d}-{n-1 \choose d-1} \choose m_0 - j - \ell}}{{{n \choose d} \choose m_0}} \\& \leq n^{(d-1)j} \left( \frac{ e \, j (d-1) n^{d-2} }{ \ell (d-2)!} \right)^{\ell} \left( \frac{m_0}{{n \choose d}-{n-1 \choose d-1} - m_0} \right)^{j+\ell} \frac{{{n \choose d}-{n-1 \choose d-1} \choose m_0}}{{{n \choose d} \choose m_0}}.
\end{align*}
We gave sharp asymptotics for the last fraction in \eqref{eqn: 7.28.2}. Here and throughout the rest of the paper, we will use $f \leq_b g$ for $f=O(g)$ when the formula for $g$ becomes too bulky. Therefore
\begin{equation*}
P(j, \ell) \leq_b (\ln n)^j \frac{e^{\omega}}{n (\ln n)^{k-1}} \left( \frac{ e \, j \, (d-1) n^{d-2} m_0}{\ell \, (d-2)! \left( {n \choose d}-{n-1 \choose d-1} - m_0 \right)} \right)^\ell;
\end{equation*}
whence
\begin{equation*}
P(j, \ell) \leq_b \frac{e^{\omega}}{n} \left( C \, \frac{\ln n}{n}  \right)^{\ell},
\end{equation*}
for $C=2(k-1)(d-1)$ (independent of $j \leq k-1$ and $\ell \geq 1$). Thus
\begin{equation*}
\sum_{j=0}^{k-1} \sum_{\ell \geq 1} E[X(j, \ell)] \leq_b e^{\omega} \sum_{\ell \geq 1} \left(  C \, \frac{\ln n}{n} \right)^{\ell} \leq_b e^{\omega} \frac{\ln n}{n} \to 0,
\end{equation*}
which completes the proof of the lemma.
\end{proof}

\begin{lemma}\label{cor: 7.28}
W.h.p. each vertex of $H_d(n, \tau_k)$ has a quasi-disjoint set of incident edges with size at least $k$. 
\end{lemma}

\begin{proof}
This lemma is trivially true for $k=1$. Suppose that $k \geq 2$. Let $A_n$ be the event that $H_d(n, \tau_k)$ has a vertex that does {\em not} have a quasi-disjoint set of edges with size at least $k$; we wish to show that $P(A_n) \to 0$. Let $m_0, m_1$ be as defined in Lemma \ref{lemma: 6.5.1}. We have proved that w.h.p. $\tau_k \in [m_0, m_1]$ and that $H_d(n,m_0)$ does not have vertices of degree less than $k-1$. Further, w.h.p. the number of degree $k-1$ vertices in $H_d(n, m_0)$ is less than $\tfrac{3 e^{\omega}}{2(k-1)!}$ (Lemma \ref{lemma: 6.5.1}). In addition, w.h.p. $H_d(n, m_0)$ has the two properties of the previous lemma (Lemma \ref{lemma: 6.5.2}). Let $B_n$ be the intersection of these four likely events. To prove the lemma, it suffices to show that $P(A_n \cap B_n) \to 0$. 

Let $\tilde{V}_0$ be the vertex set of vertices of degree $k-1$ in $H_d(n, m_0)$. Note that
\begin{align*}
P(A_n \cap B_n) &= \sum_{V_0 \subset [n], \, |V_0| \leq 3 e^{\omega}/(2(k-1)!)} P(A_n\cap B_n \cap \{\tilde{V}_0 = V_0\}).
\end{align*}
On the event that $A_n$ and $B_n$ occur and $\tilde{V}_0=V_0$, necessarily some edge $e_m$ is added in the hypergraph process at some step $m \in [m_0, m_1]$ such that $e_m$ includes both a vertex $v \in V_0$ and one of $v'$s $(k-1)(d-1)$ neighbors in $H_d(n, m_0)$. Thus 
\begin{align*}
P(A_n \cap B_n \cap \{\tilde{V}_0 = V_0\}) &\leq \sum_{m=m_0}^{m_1} \frac{{|V_0| \choose 1}{(k-1)(d-1) \choose 1}{n-2 \choose d-2}}{{n \choose d}-m} P(\tilde{V}_0=V_0) \\& \leq_b (m_1-m_0) e^{\omega} \frac{{n-2 \choose d-2}}{{n \choose d}-m_1} P(\tilde{V}_0 = V_0) \leq_b \frac{\omega e^{\omega}}{n} P(\tilde{V}_0 = V_0).
\end{align*}
 Therefore
\begin{equation*}
P(A_n \cap B_n) \leq_b \frac{\omega e^{\omega}}{n} \sum_{V_0 \subset [n], \, |V_0| \leq 3 e^{\omega}/(2(k-1)!)} P(\tilde{V}_0=V_0) \leq \frac{\omega e^{\omega}}{n} \to 0.
\end{equation*}

\end{proof}

\section{W.h.p. $H_d(n, \tau_k)$ is $k$-connected}

Now that we have sufficient knowledge about the structure of $H_d(n, m_0)$ and low-degree vertices in $H_d(n, \tau_k)$, we can prove our main Theorem.

{
\renewcommand{\thetheorem}{\ref{theorem: 1.25}}
\begin{theorem}
W.h.p. $H_d(n, \tau_k)$ is $k$-connected. In short, w.h.p. $\tau_k = T_k$.
\end{theorem}
\addtocounter{theorem}{-1}
}

\begin{proof}
Let $m'_i = \frac{n}{d} \left( \ln n + (k-1)\ln \ln n + (-1)^{i+1} \ln \ln \ln n \right)$, for $i=0,1.$ By Lemma \ref{lemma: 6.5.1}, we have shown that w.h.p. $\tau_k \in [m'_0, m'_1]$ and by Lemma \ref{cor: 7.28}, each vertex of $H_d(n, \tau_k)$ has a quasi-disjoint edge set of size at least $k$. Further, the property ``whichever $k-1$ vertices are deleted, there remains a giant component which includes all but up to $\ln n$ leftover vertices," denoted $\mathcal{Q}$, is an increasing property (closed under the addition of edges). Therefore, by Lemma \ref{lemma: 7.28.1}, $H_d(n, \tau_k)$ has property $\mathcal{Q}$ as well. To prove this theorem, it suffices to show that the probability that these three likely events hold yet $\tau_k < T_k$ tends to zero. 

To this end, for $m \in [m'_0, m'_1]$, let $C_m$ be the event that $H_d(n, m)$ is not $k$-connected, but each vertex has a quasi-disjoint edge set of size at least $k$ and $H_d(n, m)$ has property $\mathcal{Q}$. To prove this theorem, it suffices to prove that $P(\cup C_m) \to 0$. We will in fact show that
\begin{equation}\label{eqn: 7.28.1}
P(C_m) \leq_b \frac{(\ln n)^{dk+k+1}}{n^{d-1}},
\end{equation}
uniformly over $m \in [m'_0, m'_1]$. In this case, $P(\cup C_m) \leq_b \frac{(\ln n)^{dk+k+2}}{n^{d-2}} \to 0,$ as desired. All that remains is to prove the bound \eqref{eqn: 7.28.1}.

On the event $C_m$, there are $k-1$ vertices, $w_1, ..., w_{k-1}$ such that upon their deletion, there is a component of size $n'-s$ for some $s \in [1, \ln n)$. In fact, since each remaining vertex must have at least one incident edge, we must have that $s \geq d$. Let $S$ be the set of vertices not in this large component. By the union bound over all $k-1$ element sets of $[n]$ and sets $S$, $|S|=s$, as well as symmetry, we have that 
\begin{equation*}
P(C_m) \leq {n \choose k-1} \sum_{s=d}^{\ln n} {n-(k-1) \choose s} P_s,
\end{equation*}
where $P_s$ is the probability that each vertex of $H_d(n,m)$ has a quasi-disjoint edge set of size at least $k$ and that after the deletion of $\{n'+1, ..., n'+(k-1)=n\}$ from $H_d(n, m)$, the vertices $[n'-s]$ form a component; in this case, $S=\{n'-s+1, ..., n'\}$. We now turn to showing that $P_s$ tends to zero sufficiently fast. 

Suppose that $H$ is some hypergraph in the event corresponding to $P_s$. After the deletion of the $k-1$ vertices, we know that vertex $w:=\{n'\}$ from $S$ has at least one incident edge, which necessarily must reside completely within $S$. Further, before deletion, any incident edge to $w$ must be completely contained within $S$ or this edge must contain one of the $k-1$ to-be-deleted vertices. Moreover, there are at least $k$ edges incident to $w$ before the deletion. Therefore
\begin{equation*}
P_s \leq \sum_{i=1}^k P_s(i),
\end{equation*}
where $P_s(i)$ is the corresponding probability to when there are (at least) $i$ incident edges to $w$ contained within $S$ and (at least) $k-i$ incident edges to $w$ that contain at least one of the to-be-deleted vertices. To bound the number of hypergraphs contributing to $P_s(i)$, we choose $i$ potential edges within $S$ containing $w$, $k-i$ potential edges that include $w$ and at least one to-be-deleted vertex; then we choose the remaining $m-k$ edges among all potential edges except those that include a vertex of $S$ and $d-1$ vertices of $[n'-s]$ (which necessarily can not be present). Note that these last chosen edges can include $w$ as well. Therefore
\begin{equation*}
P_s(i) \leq {{s-1 \choose d-1} \choose i} {{1 \choose 1}{k-1 \choose 1}{n-2 \choose d-2} \choose k-i} {{n \choose d}-{s \choose 1}{n'-s \choose d-1} \choose m-k} \frac{1}{{{n \choose d}\choose m}}.
\end{equation*}
First, we use trivial bounds on the first two binomial terms. Then we use the inequality ${N - \ell \choose j} \leq {N \choose j} e^{- j \ell / N}$. Namely, note that
\begin{align*}
P_s(i) &\leq s^{di} \, k^k \, n^{(d-2)(k-i)} \left( \frac{m}{{n \choose d}-s{n'-s\choose d-1}-m}\right)^k \frac{{{n \choose d}-s{n'-s \choose d-1} \choose m}}{{{n \choose d} \choose m}} \\ & \leq_b s^{dk} \, n^{(d-2)(k-1)} \left( \frac{\ln n}{n^{d-1}} \right)^k \exp \left( - \frac{s{n'-s \choose d-1} m}{{n \choose d}}\right).
\end{align*}
Further, for $s \leq \ln n$, we have that 
\begin{equation*}
\frac{s {n'-s \choose d-1}m}{{n \choose d}} = \frac{s \, d \, m}{n} + O\left( \frac{(\ln n)^2}{n} \right) \geq \frac{s \, d \, m'_0}{n} + o(1).
\end{equation*}
Hence
\begin{align*}
P_s(i) &\leq_b (\ln n)^{dk} n^{(d-2)(k-1)} \left( \frac{\ln n}{n^{d-1}}\right)^k \left( \frac{\ln \ln n}{n (\ln n)^{k-1} } \right)^s \\ &\leq (\ln n)^{dk+k} \, n^{2-d-k} \, \left( \frac{\ln \ln n}{n (\ln n)^{k-1}}\right)^s, 
\end{align*}
which no longer depends on $i$. Therefore
\begin{equation*}
P_s \leq_b (\ln n)^{dk+k} \, n^{2-d-k} \, \left( \frac{\ln \ln n}{n (\ln n)^{k-1}}\right)^s,
\end{equation*}
and
\begin{align*}
P(C_m) &\leq_b n^{k-1} \sum_{s=d}^{\ln n} \frac{n^s}{s!} (\ln n)^{dk+k} \, n^{2-d-k} \, \left( \frac{\ln \ln n}{n (\ln n)^{k-1}}\right)^s.
\end{align*}
Now taking on this sum, we find that
\begin{align*}
P(C_m) \leq_b \frac{(\ln n)^{dk+k}}{n^{d-1}} \sum_{s=d}^{\ln n} \frac{1}{s!} \left( \frac{\ln \ln n}{(\ln n)^{k-1}}\right)^s \leq \frac{ (\ln n)^{dk+k}}{n^{d-1}} \exp \left(  \frac{\ln \ln n}{(\ln n)^{k-1}}\right),
\end{align*}
and we find that $P(C_m) \leq_b \tfrac{(\ln n)^{dk+k+1}}{n^{d-1}}$, as desired.

\end{proof}

\section{Sharp Threshold of $k$-connectivity}

As a consequence of Theorem \ref{theorem: 1.25}, for any $m$, we have that 
\begin{align}\label{eqn: 8.23.1}
P(H_d(n, m) \text{ is }k\text{-connected})=P(T_k \leq m) &=P(\tau_k \leq m)+o(1) \nonumber \\&= P(\text{min-deg } H_d(n, m) \geq k)+o(1).
\end{align}

We use this fact to determine the probability that $H_d(n, m)$ and $H_d(n, p)$ is $k$-connected in the critical window.

{
\renewcommand{\thetheorem}{\ref{cor: 1.25}}
\begin{corollary}
\textbf{(i)} Let $m=\frac{n}{d}\left(\ln n + (k-1)\ln \ln n + c_n\right)$, where $c_n \to c \in \mathbb{R}$. W.h.p. $H_d(n, m)$ is $(k-1)$-connected, but not $(k+1)$-connected. Further the probability that $H_d(n, m)$ is $k$-connected tends to $e^{-e^{-c}/(k-1)!}$. \\
\textbf{(ii)} Let $p=(d-1)! \frac{\ln n + (k-1)\ln \ln n + c_n}{n^{d-1}}$, where $c_n \to c \in \mathbb{R}$. W.h.p. $H_d(n, p)$ is $(k-1)$-connected, but not $(k+1)$-connected. Further the probability that $H_d(n, p)$ is $k$-connected tends to $e^{-e^{-c}/(k-1)!}$. 
\end{corollary}
\addtocounter{theorem}{-1}
}

\begin{proof}
\textbf{(i)} First, note that w.h.p. $\tau_{k-1} < m$ and $\tau_{k+1} > m$ by Lemma \ref{lemma: 6.5.1}. Therefore, by Theorem \ref{theorem: 1.25}, w.h.p. $H_d(n, m)$ is $(k-1)$-connected, but not $(k+1)$-connected. In the lemma following this proof, we show that  $X$, the number of vertices of degree $k-1$ in $H_d(n, m)$ is asymptotically Poisson with parameter $e^{-c}/(k-1)!$. Thus
\begin{equation*}
P(\text{min-deg} \, H_d(n, m) \geq k) = P(Poi(e^{-c}/(k-1)!) = 0)+o(1) = e^{-e^{-c}/(k-1)!}+o(1).
\end{equation*}
Using the equation \eqref{eqn: 8.23.1} finishes off the proof.

\textbf{(ii)} This part will be proved from \textbf{(i)} using a standard conversion technique similar to the one used in Lemma \ref{lemma: 7.28.1}. Since the number of edges in $H_d(n,p),$ denoted $e(H_d(n,p))$, is binomially distributed on $N:={n \choose d}$ trials with success probability $p$, we have that
\begin{align*}
e(H_d(n,p)) = Np + O_p\left( \sqrt{N p (1-p) } \right) = \frac{n}{d} \left( \ln n + (k-1) \ln \ln n + c_n \right) + O_p\left(\sqrt{n \ln n}\right).
\end{align*}
Therefore, if $m_{+,-} := \frac{n}{d} \left( \ln n + (k-1) \ln \ln n + c_n^{+,-}  \right),$ where $c_n^{+,-} = c_n \pm \ln n/\sqrt{n}$, then w.h.p. $m_{-} \leq e(H_d(n,p)) \leq m_+$. Using this fact, \eqref{eqn: 3.4.1} becomes
\begin{equation*}
P(H_d(n,p) \in \mathcal{A}) = \sum_{m=m_{-}}^{m_+} P(H_d(n,m) \in \mathcal{A}) P(e(H_d(n,p))=m) + o(1).
\end{equation*}
Notice that $c_n^{+,-}\to c$. By part {\bf (i)}, if $\mathcal{A}$ is $\{(k-1)\text{-connected}\}$ or $\{\text{not }(k+1)\text{-connected}\}$, then $P(H_d(n,m_{+,-})\in \mathcal{A}) \to 1$; also, if $\mathcal{A}=\{k\text{-connected}\}$, then $P(H_d(n,m_{+,-})\in \mathcal{A}) \to e^{-e^{-c}/(k-1)!}$. To finish off the proof that $P(H_d(n,p) \in \mathcal{A})$ has the same limits,  we will use the fact that these properties are monotone.

Now an {\it increasing } ({\it decreasing}) property is a property that is closed under the addition (deletion resp.) of edges. For any increasing property $\mathcal{A}$, we have that $P(H_d(n,m) \in \mathcal{A}) \leq  P(H_d(n,m') \in \mathcal{A})$ for any $m \leq m'$. This fact is obvious when you consider $H_d(n,m)$ and $H_d(n,m')$ to be snapshots of the random hypergraph process $\{H_d(n,\mu)\}_{\mu=0}^{N}$.  Moreover, if $\mathcal{A}$ is an increasing property, then
\begin{equation*}
P(H_d(n,m_-) \in \mathcal{A})+o(1) \leq P(H_d(n,p) \in \mathcal{A}) \leq P(H_d(n,m_+)\in \mathcal{A}) + o(1);
\end{equation*}
further, if $\mathcal{A}$ is decreasing, then the inequalities above are reversed. Consequently, for a monotone property $\mathcal{A}$ such that both $P(H_d(n,m_{+,-})\in \mathcal{A})$ tend to the same number, then $P(H_d(n,p) \in \mathcal{A})$ does as well. 
\end{proof}

\begin{lemma}
Let $m=\frac{n}{d} (\ln n + (k-1) \ln \ln n +c_n )$, where $c_n \to c \in \mathbb{R}$. W.h.p. the number of vertices of degree $k-1$, denoted by $X$, converges in distribution to a Poisson random variable with parameter $e^{-c}/(k-1)!$.
\end{lemma}

\begin{proof}
We prove this lemma using the method of moments (see \cite{bollobas book} for a description of this method). In order to prove the lemma, it suffices to show that for each $r \in \mathbb{N}$ (fixed), 
\begin{equation*}
\lim_{n \to \infty} E[(X)_r] = \left( \frac{e^{-c}}{(k-1)!}\right)^r.
\end{equation*}
To compute the $r$'th factorial moment, note that
\begin{equation*}
E[(X)_r] = E_0 + E_1 + \ldots + E_{(k-1)r},
\end{equation*}
where $E_j$ is the expected number of ordered $r$-tuples of vertices of degree at most $k-1$ such that there are exactly $(k-1) \, r - j$ edges containing at least one of these vertices. We will see that the terms other than $E_0$ are negligible. 

Let's first consider $E_0$. Since the number of edges is $(k-1) r$, each of these $r$ vertices have degree $k-1$ and must necessarily not be adjacent; so
\begin{equation*}
E_0 = (n)_r { {n-r \choose d-1} \choose k-1}^r \frac{{{n-r \choose d}\choose m-r(k-1)}}{{{n \choose d} \choose m}} = \left( 1 + O\left( n^{-1} \right) \right) n^r \left( \frac{n^{(k-1)(d-1)}}{(k-1)! ( (d-1)! )^{k-1} } \right)^r \frac{{{n-r \choose d}\choose m-r(k-1)}}{{{n \choose d} \choose m}}.
\end{equation*}
Taking on this last factor, we have that
\begin{equation*}
{{{n-r \choose d} \choose m-r(k-1)}} = {{{n-r \choose d} \choose m}} \frac{(m)_{r(k-1)}}{\left( {n-r \choose d} - m + r(k-1) \right)_{r(k-1)}}.
\end{equation*}
By sharply approximating this last fraction, we have that
\begin{equation*}
E_0 = \left( 1 + O \left( \frac{\ln \ln n}{\ln n} \right) \right) \left( \frac{ n \, (\ln n)^{k-1} }{(k-1)!} \right)^r \frac{{{n-r\choose d} \choose m}}{{{n \choose d} \choose m}}.
\end{equation*}
Note that
\begin{equation*}
{{n-r \choose d} \choose m} = \frac{1}{m!} \left({n-r \choose d} \right)^m \prod_{i=0}^{m-1} \left(1 - \frac{i}{{n-r \choose d}}\right) = \frac{1}{m!} \left( {n-r \choose d}\right)^m \left( 1 + O\left( \frac{m^2}{n^d}\right)\right).
\end{equation*}
We also have $d \geq 3$ so that
\begin{equation*}
{{n-r \choose d} \choose m} = \frac{1}{m!} {n-r \choose d}^m \left( 1 + O\left( \frac{(\ln n)^2}{n} \right) \right);
\end{equation*}
similarly, we have that
\begin{equation*}
{{n \choose d} \choose m} = \frac{1}{m!} { n \choose d}^m \left( 1 + O\left(\frac{(\ln n)^2}{n} \right) \right).
\end{equation*}
Further 
\begin{equation*}
\left(\frac{{n-r \choose d}}{{n \choose d}} \right)^m = \left( \left( \frac{n-r}{n} + O\left(\frac{1}{n^2}\right) \right)^d \right)^m = \exp \left( d \, m \left( - \frac{r}{n} + O(n^{-2})\right) \right),
\end{equation*}
and
\begin{equation}\label{eqn: 1.28.1}
\frac{{{n-r \choose d} \choose m}}{{{n \choose d} \choose m}} = \left( 1 + O\left( \frac{(\ln n)^2}{n} \right) \right) e^{- r d m / n} = \left( 1 + O\left( \frac{(\ln n)^2}{n} \right) \right) \left( \frac{e^{-c_n}}{n (\ln n)^{k-1} } \right)^r.
\end{equation}
Thus
\begin{equation*}
E_0 = \left( 1 + O \left( \frac{\ln \ln n}{\ln n} \right) \right) \left( \frac{e^{-c_n}}{(k-1)!} \right)^r.
\end{equation*}

Now we turn to $E_j$, for $j \geq 1$. Note that $E_j$ is less than the expected number of such $r$-tuples where these $r$ vertices have exactly $r(k-1)-j$ adjacent edges (we dropped the degree condition). Then, for $j \geq 1$, we have that
\begin{equation*}
E_j \leq (n)_r {{n \choose d} - {n-r \choose d} \choose (k-1)r-j} \frac{{{n-r \choose d} \choose m-(k-1)r+j}}{{{n \choose d} \choose m}}.
\end{equation*}
In particular, we have that
\begin{equation*}
E_j \leq n^r \left( {n \choose d} - {n-r \choose d} \right)^{(k-1)r-j} \left( \frac{m}{{n-r \choose d} - m} \right)^{(k-1)r-j} \frac{{{n-r \choose d} \choose m}}{{{n \choose d} \choose m}}.
\end{equation*}
Using the fact that
\begin{equation*}
{n \choose d}-{n-r \choose d} = r {n-r \choose d-1} + O(n^{d-2}) \leq r n^{d-1}
\end{equation*}
along with the bound \eqref{eqn: 1.28.1}, we have that
\begin{equation*}
E_j \leq_b n^{r + (d-1)[(k-1)r - j]} \left( \frac{\ln n}{n^{d-1}} \right)^{(k-1)r - j} \left( \frac{1}{n (\ln n)^{k-1}}\right)^r = \frac{1}{(\ln n)^j},
\end{equation*}
which completes the proof of the lemma.
\end{proof}


\addcontentsline{toc}{section}{References}

\begin{thebibliography}{10}\label{sec: bibli}


\bibitem{bollobas4} Bollob\'{a}s, B. (1981). \newblock Degree sequences of random graphs, \newblock {\em Discrete Mathematics}, {\bf 33}, 1-19.

\bibitem{bollobas book} Bollob\'{a}s, B. (2001). \newblock {\em Random Graphs}, \newblock Academic Press, London.

\bibitem{bollobas thomason} Bollob\'{a}s, B., \& Thomason, A.G. (1985). \newblock Random graphs of small order, \newblock In {\em Random Graphs}, Annals of Discr. Math., pp. 47-97.

\bibitem{ER59} Erd\H{o}s, P., \& R\'{e}nyi A. (1959). \newblock On Random Graphs, \newblock {\em Publ. Math. Debrecen}, {\bf 6}, 290-297.

\bibitem{erdos renyi} Erd\H{o}s, P., \& R\'{e}nyi A. (1961). \newblock On the strength of connectedness of a random graph, \newblock {\em Acta Mathematica Hungarica}, {\bf 12(1)}, 261-267.

\bibitem{ivchenko} Ivchenko, G.I. (1973). \newblock The strength of connectivity of a random graph, \newblock {\em Theory Probab. Applics}, {\bf 18}, 396-403.

\bibitem{stepanov} Stepanov, V.E. (1970). \newblock On the probability of connectedness of a random graph $G_m(t)$, \newblock {\em Theory Probab. Applics}, {\bf 15}, 55-67.









\end{thebibliography}
\end{document}